\documentclass[11pt,reqno]{amsart}
\usepackage{amssymb,amsmath,amsthm,newlfont,enumerate}

\newtheorem{prop}{Proposition}[section]

\newtheorem{thm}[prop]{Theorem}
\newtheorem{cor}[prop]{Corollary}
\newtheorem{lem}[prop]{Lemma}

\newcommand{\cD}{{\mathcal{D}}}

\newcommand{\hH}{\mathrm{H^2 }}

\newcommand{\Hol}{\mathrm{Hol}}

\newcommand\TT{\mathbb{T}}
\newcommand\DD{\mathbb{D}}

\DeclareMathOperator{\supp}{supp}

\DeclareMathOperator{\dist}{dist}

\begin{document}
\title[Poincare's inequality and uniqueness set]{Poincare type  inequality  for Dirichlet spaces and application to the uniqueness set}

\author[K.Kellay]{Karim Kellay}
\address{ Universite d'Aix--Marseille I\\
CMI\\LATP\\ 
 39\\ rue F. Joliot-Curie\\ 13453 Marseille France}
\email{ kellay@cmi.univ-mrs.fr}
\thanks{This work was partially supported  by ANR Dynop}


\begin{abstract}
 We give an extension of Poincar\'e's  type capacitary inequality for  Dirichlet spaces and provide an application to study the  uniqueness sets on the unit circle for these spaces.
 \end{abstract}

\maketitle

\section{introduction}

Let $\DD$ be the open unit disk in the complex plane  and let $\TT=\partial \DD$ be the unit circle. For $0<\alpha\leq 1$, the Dirichlet space $\cD_\alpha$ consists of all analytic functions $f$ defined on $\DD$ such that
 $$\cD_{\alpha}(f):=\int_{\TT}\int_{\TT}\frac{|f(z)-f(w)|^2}{|z-w|^{1+\alpha}}\frac{|dz|}{2\pi}\frac{|dw|}{2\pi}<\infty.$$
The space $\mathcal{D}_\alpha$ is endowed with the  norm 
$$\|f\|_{\alpha}^{2}:=|f(0)|^2+\cD_{\alpha}(f).$$
By \cite{DH}, this norm is comparable to 
$$\sum_{n\geq 0}|\widehat{f}(n)|^2(1+n)^\alpha.$$
The classical Dirichlet space $\cD_1$ is  a subspace of the  Sobolev space $\mathrm{W}^{1,2}(\DD)$, defined as  the completion of $\mathcal{C}^1(\DD)$ under the norm 
$$\|f\|^2=\Big|\int_\DD f (z)d\mathrm{A}(z)\Big|^2+\int_\DD|\nabla f(z)|^2d\mathrm{A}(z),$$
where $d\mathrm{A}(z)$ is a normalized Lebesgue measure.   Note that the restriction of this norm to $\cD_1$, becomes 
$$\|f\|^2=|f(0)|^2+ \int_\DD|f'(z)|^2d\mathrm{A}(z),\qquad f\in \cD_1,$$
 which is equivalent to the norm of $\cD_1$.

Given $f\in \mathrm{W}^{1,2}(\DD)$,  we write  $\mathrm{Z}(f)=\{z\in\DD \text{Ê: }f(z)=0\}$, the zero set of $f$ in $\DD$. The Poincar\'e capacitary inequality in $ \mathrm{W}^{1,2}(\DD)$ gives  the precise    asymptotic behavior of the constant in Poincar\'e's inequality \cite{M1,Z, KM, AH} (see also the paper \cite{Maz} by Maz'ya and the references there). More precisely there exists a constant $c>0$ such that 
\begin{equation}\label{poincaresobolev}\int_\DD |f(z)|^2d\mathrm{A}(z)\leq \frac{c}{\text{cap}_2( \mathrm{Z}(f))}\int_\DD|\nabla f(z)|^2d\mathrm{A}(z),
\end{equation}
for all $f\in \mathrm{W}^{1,2}(\DD)$, $\|\nabla f\|_2\neq 0$, where 
$$\text{cap}_2(E)=\inf\left\{\int_\DD|\nabla \varphi|^2\text{ : } \varphi\in \mathrm{C}^\infty_0(\DD) \text{, } \varphi\geq 1 \text{ on } E\right\}$$
and $\mathrm{C}^{\infty}_0(\DD)$ is the set of all infinitely differentiable functions of compact support in $\DD$. 
Our main result in this paper is to establish  a Poincar\'e capacitary inequality  for functions in the Dirichlet spaces  with the zero set is  contained  in $\TT$ (see Theorem \ref{capacitepoincare}).  We provide a sufficient condition for a set  to be uniqueness set for  Dirichlet spaces (see Theorem \ref{unicite}). 

Let $\mathrm{X}$ be some class of analytic functions in $\DD$ and let $E$ be a subset of $\TT$. The set $E$ is said to be a uniqueness set for $X$ if, for each $f\in \mathrm{X}$ such that $f^*(\zeta):=\lim\limits_{r\to 1-}f(r\zeta)= 0$ for all $\zeta\in E$, we have $f=0$.

It is clear that $\cD_\alpha$  is contained in the Hardy space $\hH $.  So each function 
$f\in \cD_\alpha$ has non-tangential limits a.e on $\TT$.  It is known that every set $E\subset \TT$ of positive Lebesgue measure is a uniqueness set  for all functions of bounded type in $\DD$ (and therefore, for $\hH $). Carleson \cite{Ca} proved  that  a closed set  of Lebesgue measure zero $E\subset\TT$ is a uniqueness set for the  Lipschitz class if and only if $E$ is not a Carleson set ($\log  \dist(\cdot, E)\not\in \mathrm{L}^{1}(\TT)$). He also proved in the same paper that if $E$ is not a Carleson set under  capacitary condition (in particular $E$ has a positive $C_s$--capacity for some $s>0$), then $E$ is a uniqueness set for the classical Dirichlet space.  Khavin and Maz'ya \cite{KM}  have proved  that there exists a set of uniqueness of $C_s$--capacity zero for any $s>0$ for the classical Dirichlet space. The proof of Khavin and Maz'ya is  based on Poincar\'e's inequality in the Sobolev space \eqref{poincaresobolev}. However,  the Khavin--Maz'ya Theorem   does not allow to deduce the Carleson Theorem. Here, we give a generalization of Khavin--Mazya's result which works for $\cD_\alpha$ spaces, $0<\alpha\leq1$, and from it we deduce Carleson's result.  Our proof is based on a  local Poincar\'e type capacitary inequality in Dirichlet spaces (see Theorem \ref{capacitepoincare}). 
 \section{Poincare's capacitary inequality}
 \subsection{Capacity} We begin with the definition of  the classical capacity \cite{Ca,KS}.  We define the kernel  on $\TT$ by 
$$
k_{\alpha}(\xi)=\left\{
 \begin{array}{lll}
|1-\zeta|^{-\alpha}, & 0<\alpha<1,\\
|\log |1-\zeta||,& \alpha =0.
\end{array}
\right.
$$
Given a probability measure $\mu$ on $\TT$, for $0\leq \alpha < 1$, we define its $\alpha$--energy by 
 $$I_\alpha(\mu)=\iint k_{\alpha}(\zeta\overline{\xi})d\mu(\xi)d\mu(\zeta).$$
 Given a Borel subset $E$ of $\TT$, we denote by $\mathcal{P}(E)$ the set of all probability measures supported on a compact subset of $E$. We define its $C_{\alpha}$--capacity by
 $$C_{\alpha}(E)=1/\inf\{I_\alpha(\mu)\text{ : } \mu \in \mathcal{P}(E)\}.$$ 
 If $\alpha=0$, $C_0$ is called the logarithmic capacity.  Note that for a set $E\subset \TT$, $C_\alpha(E)>0$ means that there exists a Borel positive finite measure   $\mu$ supported by $E$ with finite energy
 $$\sum_{n\geq 1}\frac{|\widehat{\mu}(n)|^2}{n^{1-\alpha}}<\infty.$$
 
 Now we define   the $\mathrm{L}^2$--capacity introduced by Meyers \cite{M} see also \cite{AH, AE}. For $0<\alpha\leq 1$, the harmonic Dirichlet space $\cD_\alpha(\TT)$ consists of all  functions  $f\in \mathrm{L}^2(\TT)$ such that 
 $$\cD_\alpha(f)<\infty$$
 with the norm
 $$\|f\|^{2}_{\cD_\alpha(\TT)}=\|f\|^2_{\mathrm{L}^2(\TT)}+\cD_\alpha(f).$$
This norm is comparable to 
$$\sum_{n\geq 0}|\widehat{f}(n)|^2(1+|n|)^\alpha.$$
 We have 
 $\widehat{k_{1-\frac{\alpha}{2}}}(n)\sim |n|^{-\frac{\alpha}{2}}$ as $n\to \pm\infty$ and so $\|k_{1-\frac{\alpha}{2}}\star f\|_\alpha $ is   comparable to $\|f\|_{\mathrm{L}^2(\TT)}$ for all $f\in \mathrm{L}^2(\TT)$. Hence 
  $$\cD_\alpha(\TT)=\left\{k_{1-\frac{\alpha}{2}}\star f\text{ : }f\in \mathrm{L}^2(\TT)\right\}.$$ 
   For any set $E\subset \TT$ we define the  $C_{\alpha,2}$ capacity  by
 $$C_{\alpha,2}(E):=\inf\left\{\|f\|^{2}_{\mathrm{L}^2(\TT)}\text{ : } f\in \mathrm{L}^2(\TT)\text{ , } f\geq 0\text{ , } k_{1-\frac{\alpha}{2}}\star f\geq 1 \text{ on }E\right\}.$$
 This capacity is  comparable to 
 $$ \inf\left\{\|f\|_{\cD_\alpha(\TT)}^{2}\text{ : } f\in \cD_\alpha(\TT)\text{ , } f\geq 0 \text{ , }f\geq 1 \text{ on }E\right\}.$$
 Furthermore $C_{\alpha,2}(E)$ is comparable  
  to the classical capacity $C_{1-\alpha}$,  where the implied constants depend only on $\alpha$, see  \cite{M} Theorem 14, \cite{AH} Theorem 2.5.5. 
 We finally mention the results of Beurling \cite{B} and Salem Zygmund \cite{KS, Ca,C} about the boundary behavior for the functions of the Dirichlet spaces:   if  $f\in \cD_\alpha$, we write $f^*(\xi)=\lim\limits_{r\to 1-}f(r\xi) $,  then $f^*$ exists $C_{1-\alpha}$--q.e on $\TT$, that is
   $$C_{1-\alpha}(\{\zeta\in \TT\text{ : } f^*(\zeta) \text{ does not  exist}\})=0.$$ 
Note that if $E$ is a closed set such that $C_{1-\alpha}(E)=0$, then there exists a function $f\in \cD_\alpha$ with $f^*(\zeta)=0$ on $E$ (see \cite{C}).

 \subsection{Poincar\'e's capacitary inequality for the Dirichlet spaces} Let $I$, $J$ be two  open 
 arcs of $\TT$ and $f$ be a function. We set

 $$\cD_{I,J,\alpha}(f)=\int_I\int_J\frac{|f(z)-f(w)|^2}{|z-w|^{1+\alpha}}\frac{|dz|}{2\pi}\frac{|dw|}{2\pi}, $$
and 
$$\cD_{I,\alpha}(f)=\cD_{I,I,\alpha}(f).$$
 We begin with a simple extension lemma.
\begin{lem}\label{fextention} Let $0<\gamma<1$ and let  $I=(e^{-i\theta},e^{i\theta})$ with $\theta<\gamma\pi/2$. Let 
$f\in \cD_\alpha$, then there exists a function $\widetilde{f}$ coincide with $f$ in $I$ and such that 
\begin{equation}
\label{ftilde}\cD_{J,\alpha}(\widetilde{f})\leq c\;\cD_{I,\alpha}(f), 
\end{equation}
where $J=(e^{-2i\theta/(1+\gamma)},e^{2i\theta/(1+\gamma)})$ and $c$   an absolute constant.
\end{lem}
 \begin{proof}
 Let  $\widetilde{f}$ be such that  
   $$
   \widetilde{f}(e^{it}) =\left\{\begin{array}{ll}
   f (e^{it})& \quad e^{it} \in I,\\
    f(e^{i\frac{3\theta-t}{2}}) &\quad  e^{it}\in L:=(e^{i\theta},e^{2i\theta/(1+\gamma)}),\\
     f(e^{-i\frac{3\theta+ t}{2} })  &\quad e^{it}\in R:=(e^{-2i\theta/(1+\gamma)},e^{-i\theta}).
     \end{array}
     \right.
     $$   
We write 
\begin{multline*}
\cD_{J,\alpha}(\widetilde{f})=
\cD_{I,\alpha}({f})+\cD_{L,\alpha}(\widetilde{f})+\cD_{R,\alpha}(\widetilde{f})\\
+2\;\cD_{I,L,\alpha}(\widetilde{f})+2\;\cD_{I,R,\alpha}(\widetilde{f})+2\;\cD_{L,R,\alpha}(\widetilde{f}).
\end{multline*}
If $u,v\in (\frac{1+3\gamma}{2(1+\gamma)}\theta,\theta)$, then $\pi>|2u-2v|\geq|u-v|$. By change of variable, we get
$$
\cD_{L,\alpha}(\widetilde{f})={4}\int_{ \frac{1+3\gamma}{2(1+\gamma)}\theta}^{\theta}\int_{\frac{1+3\gamma}{2(1+\gamma)}\theta}^{\theta}
\frac{|f(e^{iu})-f(e^{iv})|^2}{|e^{i(3\theta-2u)}-e^{i(3\theta-2v)}|^{1+\alpha}}\frac{du}{2\pi}\;\frac{dv}{2\pi}
\leq4\cD_{I,\alpha}({f}).
$$
 The same inequality holds for $\cD_{R,\alpha}(\widetilde{f})$.

If $u\in (\frac{1+3\gamma}{2(1+\gamma)}\theta,\theta)$ and $t\in (-\theta,\theta)$, then 
$\pi>3\theta-2u-t\geq |u-t|$ and 
$$\cD_{I,L,\alpha}(\widetilde{f})=2\int_{-\theta}^{\theta}\int_{ \frac{1+3\gamma}{2(1+\gamma)}\theta}^{\theta}\frac{|f(e^{it})-f(e^{iu})|^2}{|e^{it}-e^{i(3\theta-2u)}|^{1+\alpha}}\frac{dv}{2\pi}\frac{dt}{2\pi}\leq 2\;\cD_{I,\alpha}({f}).$$
The same inequality holds also for $\cD_{I,R,\alpha}(\widetilde{f})$.

 If $u\in ( \frac{1+3\gamma}{2(1+\gamma)}\theta,\theta)$ and   $v\in (-\theta,-\frac{1+3\gamma}{2(1+\gamma)}\theta)$, then $\pi>(3\theta-2u)+(3\theta+2v)\geq u-v$ and 
\begin{eqnarray*}
\cD_{L,R,\alpha}(\widetilde{f})&=&{4}\int_{{ \frac{1+3\gamma}{2(1+\gamma)}\theta}}^{\theta}\int_{-\theta}^{-\frac{1+3\gamma}{2(1+\gamma)}\theta }\frac{|f(e^{iu})-f(e^{iv})|^2}{|e^{i (2\theta-2u)}-e^{-i(2\theta+2v) }|^{1+\alpha}}\frac{dv}{2\pi}\frac{du}{2\pi}\\
&\leq&4\;\cD_{I,\alpha}({f}).
\end{eqnarray*}
 Hence \eqref{ftilde} is proved. 
 \end{proof}
 Given $E\subset \TT$, we write $|E|$ for the Lebesgue measure of $E$.  We can now state the main result of this section.
 \begin{thm}\label{capacitepoincare} Suppose that $0<\gamma<1$. Let $E\subset \TT$ and   $f\in \mathcal{D}_\alpha$ be  such that  $f^*|E=0$. Then, for any open arc  $I\subset \TT$ with $|I|\leq \gamma\pi$  and any $0<\beta \leq \alpha$, 
$$\Big[\frac{1}{|I|}\int_I |f(\xi)||d\xi|\Big]^2\leq \frac{c\;|I|^{\alpha-\beta}}{C_{\beta,2}(E\cap I)}\cD_{I,\alpha}(f),$$
where $c$ is a constant depending only on $\beta$ and $\gamma$.
\end{thm}
 \begin{proof} For simplicity, we will assume that $I=(e^{-i \theta}, e^{i\theta})$ with $\theta<\gamma\pi/2$. Let $J=(e^{-2i \theta/(1+\gamma)}, e^{2i\theta/(1+\gamma)})$, $\theta_\gamma=\frac{3+\gamma}{2(1+\gamma)}\theta$ the midpoint of $(\theta, 2\theta/(1+\gamma))$ and 
 $I_\gamma= (e^{-i\theta_\gamma},e^{i\theta_\gamma})$. 
 Let $\phi$ be a positive function  on  $\TT$ such that  $\supp \phi=I_\gamma$, $\phi=1$ on $I$ and  
  $$\displaystyle |\phi(z)-\phi(w)|\leq \frac{c_\gamma}{|J|} |z-w|, \qquad z,w\in \TT.$$
where $c_\gamma$ is a constant depending only on $\gamma$.

 Now let $\widetilde{f}$ be the function given in  Lemma \ref{fextention} and set  
$$F(z)= \phi(z)\Big|1-\frac{|\widetilde{f}(z)|}{m}\Big|,\qquad z\in \TT,$$
with  $$m:=\frac{1}{|J|}\int_J|\widetilde{f}(\zeta)||d\zeta|.$$
Hence  $F\geq 0$, $F_{|E\cap I}=1$ $C_{1-\alpha}$--q.p and thus $F_{| E\cap I}=1$ $C_{1-\beta}$--q.p,     since if  $C_{1-\alpha}(A)=0$, we have $C_{1-\beta}(A)=0$.  Therefore, 
\begin{eqnarray}\label{normcap}
C_{{\beta},2}(E\cap I)&\simeq&\inf
\Big\{\|g\|_{\cD_\beta(\TT)}^2\text{ : } g\geq 0 \text{ , } g\geq 1
 \text{ $C_{{\beta},2}$--q.p on } E\cap I \Big\}\nonumber \\
&\leq& c_\beta\|F\|_{\cD_\beta(\TT)}^2,
 \end{eqnarray}
 where $c_\beta$ is a constant depending only on $\beta$.

In order to conclude, we  estimate $\|F\|_{\cD_\beta(\TT)}^2$. First,  
\begin{eqnarray}\label{inegaliteF}
\|F\|_{\cD_\beta(\TT)}^2&=&\int_\TT|F(z)|^2\frac{|dz|}{2\pi}+ \int_{\TT}\int_{\TT} \frac{|F(z)-F(w)|^2}{|z-w|^{1+\beta}}\frac{|dz|}{2\pi}\frac{|dw|}{2\pi}\nonumber\\
& \leq&  \frac{1}{m^2}\int_{J}{|m-|\widetilde{f}(z)||^2}\frac{|dz|}{2\pi}+ \int_J\int_J \frac{|F(z)-F(w)|^2}{|z-w|^{1+\beta}}\frac{|dz|}{2\pi}\frac{|dw|}{2\pi} \nonumber\\
&&+\frac{2}{m^2}\int_{z\in\TT\backslash J}\int_{w\in I_\gamma}\frac{|m-|\widetilde{f}(w)||^2}{|z-w|^{1+\beta}}\frac{|dz|}{2\pi}\frac{|dw|}{2\pi}\nonumber\\
&=&\frac{A}{2\pi m^2}+\frac{B}{4\pi^2}+\frac{C}{2\pi^2m^2}.
\end{eqnarray}
By \eqref{ftilde}, 
\begin{eqnarray}\label{inegalite1}
A &:= &\int_{J}{|m-|\widetilde{f}(z)||^2}|dz|\nonumber\\
&=&\frac{1}{|J|^2}\int_{J}\Big|\int_J (|\widetilde{f}(\zeta)|-|\widetilde{f}(z)|)|d\zeta|\Big|^2|dz| \nonumber\\
&\leq& \frac{1}{|J|}\int_{J}\int_J |\widetilde{f}(\zeta)-\widetilde{f}(z)|^2|d\zeta||dz|\nonumber\\
&\leq& c_1 \int_J\int_J  \frac{|\widetilde{f}(\zeta)-\widetilde{f}(z)|^2}{|\zeta-z|^{1+\beta}}|d\zeta||dz|\nonumber\\
&\leq& c_1 |J|^{\alpha-\beta} \cD_{J,\alpha}(\widetilde{f})\nonumber\\
&\leq&  c_2|I|^{\alpha-\beta} \cD_{I,\alpha}({f}), 
\end{eqnarray}
for some constants $c_1,c_2$ independent of $\beta$ and $\gamma$.

 If $(z,w)\in J\times J$, then 
\begin{eqnarray*}
&&|F(z)-F(w)|=\\
&&\Big|\phi(z)\Big(\big|1-\frac{|\widetilde{f}(z)|}{m}\big|-\big|1-\frac{|\widetilde{f}(w)|}{m}\big|\Big)+
(\phi(z)-\phi(w))\big|1-\frac{|\widetilde{f}(w)|}{m}\big|\Big|\\
&\leq& \frac{1}{m}|\widetilde{f}(z)-\widetilde{f}(w)|+ \frac{c_\gamma}{m}\frac{|z-w|}{ |J|}|m-|\widetilde{f}(w)||\\
&\leq  &\frac{1}{m}|\widetilde{f}(z)-\widetilde{f}(w)|+ \frac{c_\gamma}{m}\frac{|z-w|}{ |J|^2} \int_J|\widetilde{f}(\zeta)-\widetilde{f}(w)||d\zeta|.\\
\end{eqnarray*}
So, by  \eqref{ftilde} again, 
\begin{eqnarray}\label{inegalite2}
B&:=&  \int_{J}\int_{J} \frac{|F(z)-F(w)|^2}{|z-w|^{1+\beta}}|dz||dw|\nonumber \\
&\leq& \frac{2}{m^2} \int_{J}\int_{J} \frac{|\widetilde{f}(z)-\widetilde{f}(w)|^2}{|z-w|^{1+\beta}}|dz||dw|+\nonumber \\
&&\;\;\;\;\;\;\;\;\;\;\;\;\; \frac{2c_\gamma^2}{m^2|J|^4} \int_{J}\int_{J}  \Big(\int_J|\widetilde{f}(\zeta)-\widetilde{f}(w)||d\zeta|\Big)^2
|z-w|^{1-\beta} |dw||dz|\nonumber\\
&\leq&\frac{2+2c_\gamma^2}{m^2} \int_{J}\int_{J}  \frac{|\widetilde{f}(\zeta)-\widetilde{f}(w)|^2}{|\zeta-w|^{1+\beta}}|d\zeta||dw|\nonumber\\
&\leq& \frac{ c_3}{m^2}|I|^{\alpha-\beta} \cD_{I,\alpha}({f}), 
\end{eqnarray}
with $c_3$ is a constant depending only on  $\gamma$.

Finally, 
\begin{eqnarray}\label{inegalite3}
C&:=&\int_{z\in\TT\backslash J}\int_{w\in I_\gamma}\frac{|m-|\widetilde{f}(w)||^2}{|z-w|^{1+\beta}}|dz||dw|\nonumber\\
&\leq &\frac{c_4}{|J|^{1+\beta}}\int_{I_\gamma}{|m-|\widetilde{f}(w)||^2}|dw|\nonumber\\
&\leq & \frac{c_4}{|J|^{2+\beta}}\int_{I_\gamma}\Big|\int_J |\widetilde{f}(\zeta)-\widetilde{f}(w)||d\zeta|\Big|^2|dw|\nonumber\\
&\leq &\frac{c_4 }{|J|^{1+\beta}}\int_{I_\gamma}\int_J |\widetilde{f}(\zeta)-\widetilde{f}(w)|^2|d\zeta||dw|\nonumber\\
&\leq& 
 c_4 \iint_{J\times J} \frac{|\widetilde{f}(\zeta)-\widetilde{f}(w)|^2}{|\zeta-w|^{1+\beta}}|d\zeta||dw|\nonumber\\
&\leq& c_5|I|^{\alpha-\beta} \cD_{I,\alpha}({f}), 
\end{eqnarray}
with  $c_4, c_5$ independent of  $\gamma, \beta$.

By  \eqref{inegalite1}, \eqref{inegalite2} and \eqref{inegalite3},  we see that 
\begin{equation}\label{normfinal}
\|F\|_{\cD_\beta(\TT)}^2\leq \frac{c_6}{m^2}|I|^{\alpha-\beta} \cD_{I,\alpha}({f}),
\end{equation}
with $c_6$ depending only on $\gamma$. Since  
$$m\asymp\frac{1}{|I|}\int_I |f(\xi)||d\xi|,$$
 combining \eqref{normcap} and \eqref{normfinal}, we get 
$$
C_{{\beta},2}(E\cap I)\leq 
c \Big[\frac{1}{|I|}\int_I |f(\xi)||d\xi|\Big]^{-2} |I|^{\alpha-\beta}   \cD_{I,\alpha}({f}), 
$$
where $c$  depending only on $\beta$ and $\gamma$, and the proof is complete.
\end{proof}

\section{set of uniqueness for Dirichlet spaces}
A special case of the theorem  ($\beta=1$ in Theorem \ref{unicite}) was obtained by Khavin and Maz'ya \cite{KM}   for the classical Dirichlet space ($\alpha=1$). Here we give the generalization of their result  in the Dirichlet spaces, including   the classical case. 
\begin{thm}\label{unicite} Let  $E$ be a Borel subset of  $\TT$  of  Lebesgue measure zero. We assume that there exists a family  of pairwise disjoint open arcs  $(I_n)$ of $\TT$ such that $E\subset \bigcup\limits_n I_n$. Suppose that there exists  $0<\beta\leq\alpha$ such that 
$$
\sum_{n}|I_n|\log\frac{|I_n|^{1+\alpha-\beta}}{C_{1-\beta}(E\cap I_n)}=-\infty, 
$$
then $E$ is a uniqueness set for $\cD_\alpha$.
\end{thm}
\begin{proof} 
Since $|E|=0$, we can assume that there is $\gamma\in (0,1)$ such that $\sup_n|I_n|\leq \gamma\pi$. Let  $f\in \cD_\alpha$ be such that  $f^*|E=0$. We set $\mathcal{I}=\sum_n |I_n|$. 
Since  $(I_n)$ are disjoint, $C_{1-\beta}$ is  comparable to $C_{{\beta},2}$ . Then Theorem \ref{capacitepoincare} and Jensen inequality  give
\begin{eqnarray*}
2 \int_{\bigcup I_n}\log|f(\xi)|d\xi|& \leq& \sum_n |I_n|\log \Big(\frac{1}{|I_n|}\int_{I_n}|f(\xi)|d\xi|\Big)^2\\
 &\leq &\sum_n |I_n|\log\Big( \frac{c |I_n|^{\alpha-\beta}}{C_{1-\beta}(E\cap I_n)} \cD_{I_n,\alpha}({f})\Big)\\
&= &\sum_n |I_n| \log \frac{|I_n|^{1+\alpha-\beta}}{C_{1-\beta}(E\cap I_n)} +\mathcal{I}\sum_n \frac{|I_n|}{\mathcal{I}} \log (c\cD_{I_n,\alpha}({f}))\\
&\leq &\sum_n |I_n| \log \frac{|I_n|^{1+\alpha-\beta}}{C_{1-\beta}(E\cap I_n)} +\mathcal{I}\log\Big(\frac{c}{ \mathcal{I}}\sum_n\cD_{I_n,\alpha}({f}) \Big)\\
&\leq& \sum_n |I_n| \log \frac{|I_n|^{1+\alpha-\beta}}{C_{1-\beta}(E\cap I_n)}+ {\mathcal{I}}\log \Big(\frac{c}{ \mathcal{I}}\|f\|^{2}_{\alpha} \Big)=-\infty.
\end{eqnarray*}
By Fatou Theorem we obtain $f=0$, which finishes the proof.
\end{proof}

The following result was obtained by Carleson \cite{C} for the classical Dirichlet space. A  generalization of his Theorem was given by Preobrazhenskii in \cite{P} and by  Pau and Pelaez in \cite{PP} for the Dirichlet spaces $\mathcal{D}_\alpha$ with $0<\alpha<1$. Here we give another proof of this generalization. 
\begin{cor}\label{unicite1}. Let  $E$ be a closed subset of  $\TT$  of Lebesgue measure zero.  Let $0<\beta<\alpha\leq 1$.  Assume that there exists  $m>0$ such that  for each interval $I\subset \TT$ centered   at a point of $E$, 
\begin{equation}\label{regularite}
C_{1-\beta}(E\cap I)\geq m |I|. 
\end{equation} 
Then $E$ is a uniqueness set for $\cD_\alpha$ if and only if 
\begin{equation}\label{carlesonset}\sum_n|I_n|\log|I_n|=-\infty,
\end{equation}
 where $(I_n)_n$ are the complementary intervals of  $E$.
\end{cor}
\begin{proof}
Note that $\mathcal{A}^{1}(\DD):=\Hol(\DD)\cap \mathcal{C}^1(\overline{\DD})\subset \cD_\alpha$. If $E$ is a uniqueness set for $\cD_\alpha$, then $E$ is a uniqueness set for $\mathcal{A}^1(\DD)$ and thus $E$ is not a  Carleson set \cite{C}, i.e. $E$ has Lebesgue measure zero and satisfies \eqref{carlesonset}.

Conversely, we write $\TT\backslash E=\bigcup\limits_k I_k$ with $I_k= (e^{i\theta_{2k}},e^{i\theta_{2k+1}})$. Let  $J_{2k}$ (resp. $J_{2k+1}$) be  the open arc of length $|I_k|$ with midpoint $e^{i\theta_{2k}}$  (resp. $e^{i\theta_{2k+1}}$). By Vitali covering lemma,  there exists a sub-collection $(J_{k'})_{k'}$ of $(J_k)_k$ which is  disjoint   and satisfies $\bigcup\limits_k J_k \subset 3\bigcup\limits_{k'} J_{k'}$. Hence,   
$$\sum_{k'} |J_{k'}|\log |J_{k'}|=-\infty.$$ 
Let  $F=E\bigcap (\bigcup\limits_{k'}J_{k'})$ be the subset of  $E$ contained in $\bigcup\limits_{k'} J_{k'}$. The set  $F$ is a Borel set and, since $F\cap J_{k'}= E\cap J_{k'}$, by \eqref{regularite}, 
$$C_{1-\beta}(F\cap J_{k'})\geq  m|J_{k'}|.$$
Then for $0<\beta<\alpha\leq 1$, we obtain 
$$\sum _{k'} |J_{k'}|\log\frac{ |J_{k'}|^{1+\alpha-\beta}}{C_{1-\beta} (F\cap J_{k'})}\leq (\alpha-\beta)\sum_{k'} |J_{k'}|\log |J_{k'}|-\log m \sum_{k'}|J_{k'}|=-\infty.$$
By Theorem \ref{unicite}, the set $F$ is a set of uniqueness for $\cD_\alpha$ and so does $E$, which finishes the proof. 
\end{proof}

\subsection*{Remarks}1. A function $\varphi\in \cD_\alpha$ is called multiplier of $\cD_\alpha$ if 
$\varphi \cD_\alpha\subset  \cD_\alpha$ and we denote the set of multipliers by 
$\mathcal{M}_{\cD_\alpha}$.  Richter and Sundberg in \cite{RS} proved that a set $E$ is a  zero set of Dirichlet space $\cD_1$ if and only if  it is a zero set of $\mathcal{M}_{\cD_1}$. On the other hand if $\varphi\in\mathcal{M}_{\cD_1}$,  then by Steganga result \cite{S} Theorem 2.7.c,  we have $\cD_{I,1}(\varphi)=O(C_0 (I))$, note that $C_0(I)\asymp|\log I|^{-1}$.

2. Khavin and Maz'ya in \cite{KM} have constructed a set of uniqueness $E$ for the classical Dirichlet space  such that  $C_{1-\beta}(E)=0$ for every  $0<\beta<1$.  On the other hand, Carleson in \cite{C1} has constructed  a zero set $E$ which satisfies  \eqref{carlesonset} and $E\cap I$ has a positive logarithmic capacity for all arcs such that $E\cap I\neq \emptyset$. As in \cite{KM}, we can construct a closed set $E$ which is a set of uniqueness for $\cD_\alpha$ and such that $C_{1-\beta}(E)= 0$ for all $0<\beta<\alpha< 1$.  Let $(l_n)_{n\ge0}$ be a sequence in $(0,2\pi)$ and let $\mathcal{C}$ be the associated generalized Cantor set. Then for $0\leq s<1$,                         
$$C_{s}(\mathcal{C})=0 \iff \sum_n 2^{-n} l_{n}^{-s}=+\infty,$$
see for example \cite{AE,Ca}. 

Choose $l_n=(2^{-n}n)^\frac{1}{1-\beta}$. Then $C_{1-\beta}(\mathcal{C})=0$ and for $0<\beta<\alpha$, 
$$\sum_n 2^{-n} l_{n}^{-(1-\alpha)}= \sum_n {2^{-n\frac{\alpha-\beta}{1-\beta}}}{ n^{-\frac{1-\alpha}{1-\beta}}}<\infty.$$
Therefore, $C_{1-\alpha}(\mathcal{C})>0$. Now, consider a family of pairwise disjoint open arcs $(I_n)_n$ of $\TT$ be such that 
$$\sum_n |I_n|\log|I_n|=-\infty.$$
A possible example, $I_n=(e^{i(\log (n+1))^{-1}},e^{i(\log n)^{-1}})$, $n\geq 2$.  We reproduce the generalized Cantor set $\mathcal{C}$ in each $I_n$, which will be  denoted by $\mathcal{C}_n$. Therefore,   
$$C_{1-\alpha}(\mathcal{C}_n\cap I_n)\simeq C_{1-\alpha}(\mathcal{C})  |I_n|^\alpha.$$
We set $E=\{1\}\cup \bigcup\limits_n \mathcal{C}_n$.  It is clear that  $C_{1-\beta}(E)=0$, for all  $0<\beta<\alpha$. 
Now  Theorem \ref{unicite} with  $\beta=\alpha$ gives 
$$\sum_{n}|I_n|\log \frac{|I_n|}{C_{1-\alpha}(E\cap I_n)}\simeq-\log C_{1-\alpha}(\mathcal{C})
\sum_{n}|I_n| + (1-\alpha)\sum_n|I_n|\log|I_n|=-\infty.$$
So $E$ is a set of uniqueness for $\cD_\alpha$ with $\alpha<1$.
 
 3.  Malliavin in \cite{Ma} gives  a complete characterization of the sets of uniqueness for the Dirichlet spaces involving a new notion of capacity, but it appears difficult to apply his result to particular situations (see also \cite{KK}). 
  \subsection*{\bf Acknowledgment} I would like to thank the referee for his helpful remarks specially for those regarding the proof of Theorem \ref{capacitepoincare}
  


\begin{thebibliography}{99}

  \bibitem{AH}
D. Adams; L. Hedberg, Function spaces and potential theory, Springer, Berlin, 1996. 
  
  \bibitem{AE}
 H. Aikawa, M. Essen, Potential theory: selected topics, Lecture Note Math. 1633. Springer, Berlin, 1996.

\bibitem{B}
A. Beurling, Ensembles exceptionnels, Acta. Math. 72 (1939), 1--13. 
\bibitem{C}
L. Carleson,  Sets of uniqueness for functions regular in the unit circle, Acta Math., 87 (1952) 325--345. 

\bibitem{Ca}
L. Carleson, Selected Problems on Exceptional Sets, Van Nostrand, Princeton NJ, 1967.

\bibitem{C1} 
L. Carleson, An example concerning analytic functions with finite Dirichlet integrals. Investigations on linear operators and the theory of functions, IX.  Zap. Nauchn. Sem. Leningrad. Otdel. Mat. Inst. Steklov. (LOMI)  92  (1979), 283--287, 326.
\bibitem{DH}
A. Devinatz, I. Hirshman, Multiplier transformations on $l^{2,\alpha}$, Ann. Math. 69 (1959), 575--587.

\bibitem{KS}J. P. Kahane, R. Salem. Ensembles parfaits et s\'eries trigonom\'etriques.  Hermann,  Paris, 1963.

\bibitem{KM}V.Khavin,  V. Maz'ya,, Application of the $(p,l)$--capacity to certain problems 
of theory of exceptional sets, Mat.USSR Sb., 19  (1973), 547--580 (1974). 

\bibitem{KK}V.Khavin, S. Khrushchev, Sets of uniqueness for analytic f functions with the finite Dirichlet integral, Problem 9.3,  531--535. Linear and complex analysis problem book. Lecture Note Math. 1043.   Springer, Berlin, 1984.
\bibitem{Ma} 
P. Malliavin, Sur l'analyse harmonique de certaines classes de s\'eries de Taylor, Symposia Mathematica, Vol. XXII, Academic Press, London, 1977,  71--91.
\bibitem{Maz}
V. Maz'ya, 
Conductor and capacitary inequalities for functions on topological spaces and their applications to Sobolev type imbeddings, J. Funct. Anal. 224 (2005), no.2, 408-430


\bibitem{M}
N.Meyers, A theory of capacities for potentials of functions in Lebesgue classes, Math. scand. 26 (1970), 255--292.
\bibitem{M1}
N. Meyers, 
Integral inequalities of PoincarŽ and Wirtinger type. 
Arch. Rational Mech. Anal. 68 (1978) 113--120. 
\bibitem{PP}
J. Pau, J.A. Pelaez, 
On the zeros of functions in Dirichlet spaces. Trans Amer. Math Soc., to appear. 
\bibitem{P} S.P. Preobrazhenskii, A boundary uniqueness Theorem for regular functions with a finite Dirichlet--type integral, ap. Nauchn. Sem. Leningrad. Otdel. Mat. Inst. Steklov. (LOMI)126  (1978) 180--190.
\bibitem{RS}
S. Richter, C. Sundberg, Multipliers and invariant subspaces in the Dirichlet space, J. Operator Theory 28 (1992) 167--186. 
\bibitem{S}
D. Stegenga, Multipliers of the Dirichlet space, Illinois. J. Math. 24 (1980), 113--139.
\bibitem{Z}
W. Ziemer, Weakly differentiable Functions, Sobolev space and functions of bounded variation. Springer, New--York, 1989.

\end{thebibliography}
\end{document}